\documentclass[reqno,12pt]{article}
\usepackage{amsmath,amsthm,amssymb}
\usepackage{color}

\date{}

\newtheorem{theorem}{Theorem} \newtheorem{lemma}[theorem]{Lemma}
\newtheorem{corollary}[theorem]{Corollary}
\newtheorem{proposition}[theorem]{Proposition}

\newtheorem{remark}[theorem]{Remark}

\newcommand{\norm}[1]{\left\Vert#1\right\Vert}

\newcommand{\abs}[1]{\left\vert#1\right\vert}
\newcommand{\set}[1]{\left\{#1\right\}}
\newcommand{\brac}[1]{\left(#1\right)}
\newcommand{\scalar}[1]{\left \langle #1 \right \rangle}

\newcommand{\Real}{\mathbb{R}}

\renewcommand{\H}{\mathcal{H}}
\newcommand{\Var}{\mathbb{V}ar}

\newcommand{\eps}{\epsilon}

\newcommand{\II}{\text{II}}

\def\XXint#1#2#3{{\setbox0=\hbox{$#1{#2#3}{\int}$}
\vcenter{\hbox{$#2#3$}}\kern-.5\wd0}}

\begin{document}

\title{Remarks on the KLS conjecture and Hardy-type inequalities}
\author{Alexander V. Kolesnikov\thanks{Faculty of Mathematics, National Research University Higher School of Economics, Moscow, Russia. Email: sascha77@mail.ru}, $\ $
Emanuel Milman\thanks{Department of Mathematics, Technion - Israel
Institute of Technology, Haifa 32000, Israel. Email: emilman@tx.technion.ac.il.}}
\date{}
\maketitle

\begin{abstract}
We generalize the classical Hardy and Faber-Krahn inequalities to arbitrary functions on a convex body $\Omega \subset \Real^n$, not necessarily vanishing on the boundary $\partial \Omega$. This reduces the study of the Neumann Poincar\'e constant on $\Omega$ to that of the cone and Lebesgue measures on $\partial \Omega$; these may be bounded via the curvature of $\partial \Omega$. A second reduction is obtained to the class of harmonic functions on $\Omega$. We also study the relation between the Poincar\'e constant of a log-concave measure $\mu$ and its associated K. Ball body $K_\mu$. In particular, we obtain a simple proof of a conjecture of Kannan--Lov\'asz--Simonovits for unit-balls of $\ell^n_p$, originally due to Sodin and Lata{\l}a--Wojtaszczyk. 
\end{abstract}

\section{Introduction}

Given a compact connected set $\Omega$ with non-empty interior in Euclidean space $(\Real^n,\abs{\cdot})$ ($n \geq 2$) and a smooth function $f$  on $\Omega$ vanishing on $\partial \Omega$, a version of the classical Hardy inequality (e.g. \cite{GM}) states that:
\begin{equation} \label{eq:Hardy0}
\int_{\Omega} f^2 dx \le \frac{4}{n^2} \inf_{x_0 \in \Real^n} \int_{\Omega} |x - x_0|^2 |\nabla f|^2 dx .
\end{equation}
The classical Faber--Krahn inequality (e.g. \cite{Benguria-Lambda1Survey}) states that under the same conditions:
\begin{equation} \label{eq:FK}
\int_{\Omega} f^2 dx \leq P^D_{\Omega^*} \int_{\Omega} |\nabla f|^2 dx ,
\end{equation}
where $P^D_{\Omega^*}$ is the best constant in the above inequality under the same conditions with $\Omega = \Omega^*$, the Euclidean Ball having the same volume as $\Omega$. $P^D$ is called the Poincar\'e constant with \emph{zero Dirichlet boundary conditions}; it is elementary to verify that $P^D_{\Omega^*} \simeq \frac{1}{n} \abs{\Omega^*}^{2/n}$ (see Remark \ref{rem:Watson} for more precise information). 

In this note we explore what may be said when $f$ does not necessarily vanish on the boundary, and develop applications for estimating the Poincar\'e constant with Neumann boundary conditions. Here and elsewhere, we use $\abs{M}$ to denote the $k$-dimensional Hausdorff measure $\H^{k}$ of the $k$-dimensional manifold $M$, and $A \simeq B$ to denote that $c \leq A/B \leq C$, for some universal numeric constants $c , C>0$. All constants $c,c',C,C',C_1,C_2$, etc. appearing in this work are positive and universal, i.e. do not depend on $\Omega$, $n$ or any other parameter, and their value may change from one occurrence to the next.

Let $\lambda_{\Omega}$ denote the uniform (Lebesgue) probability measure on $\Omega$, and let $P^N_{\Omega}$ denote the Poincar\'e constant of $\Omega$, i.e. the best constant satisfying:
\begin{equation} \label{eq:SG}
\Var_{\lambda_{\Omega}} f \le P^N_{\Omega} \int_{\Omega} |\nabla f|^2 d \lambda_{\Omega} \;\;\; \forall \text{ smooth } f : \Omega \rightarrow \Real ~,
\end{equation}
without assuming any boundary conditions on $f$. 
Here and throughout $\Var_\mu(f) := \int f^2 d\mu - (\int f d\mu)^2$ for any probability measure $\mu$.  
It is well-known that when $\partial \Omega$ is smooth, $1/P^N_{\Omega}$ coincides with the first non-zero eigenvalue (``spectral-gap") of the Laplacian on $\Omega$ with zero \emph{Neumann} boundary conditions, explaining the superscript \emph{$N$} in our notation for $P^N_\Omega$.  The classical Szeg\"{o}--Weinberger inequality (e.g. \cite{Benguria-Lambda1Survey}) states that $P^N_{\Omega} \geq P^N_{\Omega^*} \simeq Vol(\Omega^*)^{2/n}$. By inspecting domains with very narrow bottlenecks, or even convex domains which are very narrow and elongated in a certain direction, it is clear that without some additional information on $\Omega$, $P^N_{\Omega}$ is not bounded from above. However, a conjecture of Kannan--Lov\'asz--Simonovits \cite{KLS} asserts that on a \emph{convex} domain $\Omega$, the Poincar\'e inequality (\ref{eq:SG}) will be saturated by linear functions $f$, up to a universal constant $C>0$ independent of $n$ and $\Omega$, i.e.:
\[
P^N_{\Omega} \leq C P^{Lin}_{\Omega} ~,~ P^{Lin}_{\Omega} := \sup_{\theta \in S^{n-1}} \Var_{\lambda_{\Omega}} \scalar{\cdot,\theta} 
\]
It is easy to reduce the KLS conjecture to the case that $\Omega$ is \emph{isotropic}, meaning that its barycenter is at the origin and the variance of all unit linear functionals is $1$, i.e.:
\[
\int x_i d\lambda_{\Omega} =0 ~,~ \int x_i x_j d \lambda_{\Omega} =\delta_{ij} \;\;\;\; \forall i,j=1,\ldots,n .
\]
The conjecture then asserts that $P^N_{\Omega} \leq C$ for any convex isotropic domain $\Omega$ in $\Real^n$. 

\medskip

Given a Borel probability measure $\mu$ on $\Real^n$ (not necessarily absolutely continuous), we denote by $P^\infty_\mu$ the best constant in the following weak $L^2$-$L^\infty$ Poincar\'e inequality:
\begin{equation} \label{eq:weak-Poincare}
\Var_{\mu} f \le P^\infty_{\mu}  \norm{|\nabla f(x)|}^2_{L^\infty(\mu)} \;\;\; \forall \text{ smooth } f : \Real^n \rightarrow \Real ~.
\end{equation}
Set $P^\infty_{\Omega} := P^\infty_{\lambda_{\Omega}}$; clearly $ P^\infty_{\Omega} \leq P^N_{\Omega}$. 
In \cite{EMilman-RoleOfConvexity}, the second-named author showed that when $\Omega$ is convex, the latter inequality may be reversed:
\begin{equation} \label{eq:equiv}
 P^N_{\Omega} \leq C P^\infty_{\Omega} ~,
\end{equation}
where $C>1$ is a universal numeric constant. This reduces the KLS conjecture to the class of $1$-Lipschitz functions $f$ (satisfying $\norm{\abs{\nabla f}}_{L^\infty} \leq 1$). Another remarkable reduction was obtained by R. Eldan, who showed \cite{Eldan-StochasticLocalization} that it is essentially enough (up to logarithmic factors in $n$) to establish the conjecture for the Euclidean norm function $f(x) = \abs{x}$, but simultaneously for \emph{all} isotropic convex domains in $\Real^n$. Employing an estimate on the variance of $\abs{x}$ due to O. Gu\'edon and the second-named author \cite{GuedonEMilmanInterpolating}, it follows from Eldan's reduction that for a general convex body in $\Real^n$, $P^N_{\Omega} \leq C n^{2/3} \log(1+n) P^{Lin}_{\Omega}$. 
 
In this work, we obtain several additional reductions of the KLS conjecture. First, we obtain a sufficient condition by reducing to the study of $P^\infty_{\sigma_{\partial \Omega}}$ and $P^\infty_{\lambda_{\partial \Omega}}$, the cone and Lebesgue measures on $\partial \Omega$, respectively. In particular, it suffices to bound the variance of homogeneous functions which are $1$-Lipschitz on the boundary. This is achieved by obtaining Neumann versions of the Hardy and Faber-Krahn inequalities (\ref{eq:Hardy0}) and (\ref{eq:FK}) for general functions (not necessarily vanishing on the boundary). The parameters $P^\infty_{\sigma_{\partial \Omega}}$ or $P^\infty_{\lambda_{\partial \Omega}}$ may then be bounded using a result from our previous work \cite{KolesnikovEMilman-Reilly}, by averaging certain curvatures on $\partial \Omega$ (see Theorem \ref{thm:KLS-MainBound}). 

Second, we reduce the KLS conjecture to the class of harmonic functions. Thirdly, we consider the Poincar\'e constant of an unconditional convex body bounded by the principle hyperplanes, when a certain mixed Dirichlet--Neumann boundary condition is imposed. It is interesting to check which of the boundary conditions will dominate this Poincar\'e constant, and we determine that it is the Dirichlet ones, resulting in a Faber--Krahn / Hardy-type upper bound. 

Lastly, we reveal a general relation between the Poincar\'e constant of a log-concave measure $\mu$ and its associated K. Ball body $K_\mu$, assuming that the latter has finite-volume-ratio. In particular, we obtain a quick proof of the KLS conjecture for unit-balls of $\ell_p^n$, $p \in [1,2]$ (first established by S. Sodin \cite{SodinLpIsoperimetry}), which avoids using the concentration estimates of Schechtman and Zinn \cite{SchechtmanZinn2}. This is also extended to arbitrary $p \geq 2$ bounded away from $\infty$.

Our proofs follow classical arguments for establishing the Hardy inequality, which can be viewed as a Lyapunov function or vector-field method, in which one is searching for a vector-field whose magnitude is bounded from above on one hand, and whose divergence is bounded from below on the other. For more applications of Lyapunov functions  to the study of Sobolev-type inequalities, see \cite{CattiauxGuillin-LyapunovSurvey}.

\medskip
\emph{Acknowledgements.} We would like to thank Bo'az Klartag for his interest and fruitfull discussions.
The first-named author was supported by RFBR project 12-01-33009 and  the DFG project  CRC 701.
This study (research grant No 14-01-0056) was supported by The National Research University–-Higher School of Economics' Academic Fund Program in 2014/2015. The second-named author was supported by ISF (grant no. 900/10), BSF (grant no. 2010288), Marie-Curie Actions (grant no. PCIG10-GA-2011-304066) and the E. and J. Bishop Research Fund. 

\section{Hardy-type inequalities}

Let  $\Omega$ denote a compact connected set  in $\mathbb{R}^n$ with smooth boundary and having the origin in its interior. We denote by $\nu$ the unit exterior normal-field to $\partial \Omega$.
We denote by $\lambda_{\partial \Omega}$ the uniform probability measure on $\partial \Omega$ induced by the Lebesgue measure, i.e. $\H^{n-1}|_{\partial \Omega} / \abs{\partial \Omega}$. 
Our basic starting point is the following integration-by-parts formula. Let $g$ denote a smooth function and $\xi$ a smooth vector field on $\Omega$. Then:
\begin{equation} \label{eq:ByParts}
\int_{\Omega} div(\xi) g dx = - \int_{\Omega} \scalar{\xi , \nabla g} dx + \int_{\partial \Omega} \scalar{\xi,\nu} g \; d\H^{n-1} ~.
\end{equation}
Applying this to $g = f^2$ and using the Cauchy-Schwartz inequality (in additive form), we obtain for any positive function $\lambda$ on $\Omega$:
\[
\int_{\Omega} div(\xi) f^2 dx \leq \int_\Omega  \lambda f^2 dx + \int_{\Omega} \frac{1}{\lambda}\abs{\scalar{\xi,\nabla f}}^2 dx + \int_{\partial \Omega} \scalar{\xi,\nu} f^2 \; d\H^{n-1} ,
\]
or equivalently:
\begin{equation} \label{eq:BasicHardy}
\int_{\Omega} \brac{div(\xi) - \lambda} f^2 dx \leq  \int_{\Omega} \frac{1}{\lambda} \abs{\scalar{\xi,\nabla f}}^2 dx + \int_{\partial \Omega} \scalar{\xi,\nu} f^2 \; d\H^{n-1} .
\end{equation}

Let us apply this to several different vector fields $\xi$.

\subsection{Radial Vector Field}

In this subsection, assume in addition that $\Omega$ is star-shaped, meaning that $\Omega = \set{x \; ; \; \norm{x} \leq 1}$, where $\norm{x} := \inf \set{ \lambda > 0 ; x \in \lambda \Omega}$ denotes its associated gauge function. We denote by $\sigma_{\partial \Omega}$ the induced cone probability measure on $\partial \Omega$, i.e. the push-forward of $\lambda_{\Omega}$ via the map $x \mapsto \frac{x}{\norm{x}}$.
It is well-known and immediate to check that:
\[
\sigma_{\partial \Omega} = \frac{1}{\abs{\Omega}} \frac{\langle x, \nu \rangle}{n} \cdot \mathcal{H}^{n-1}|_{\partial \Omega} .
\]

\begin{theorem}[Hardy with Boundary] \label{thm:Hardy1}
 Let $f$ denote a smooth function on $\Omega$. Then:
\begin{equation} \label{eq:Hardy}
\Var_{\lambda_{\Omega}} f \le \frac{4}{n^2} \int_{\Omega} \langle x, \nabla f \rangle^2  d \lambda_{\Omega}
+ 2 \Var_{\sigma_{\partial \Omega}} f .
\end{equation}
\end{theorem}
\begin{proof}
Apply (\ref{eq:BasicHardy}) with $\xi(x) = x$, so that $div(\xi) = n$, and $\lambda \equiv n/2$. We obtain:
\begin{equation} \label{eq:HardyDomain}
\int_{\Omega} f^2 dx \le \frac{4}{n^2} \int_{\Omega} \langle x, \nabla f \rangle^2  dx
+ \frac{2}{n}\int_{\partial \Omega} \langle x, \nu \rangle f^2 d \mathcal{H}^{n-1}(x) .
\end{equation}
In particular, we see that (\ref{eq:Hardy}) immediately follows when $f$ vanishes on $\partial \Omega$. For general functions, we
divide (\ref{eq:HardyDomain}) by $Vol(\Omega)$ and apply the resulting inequality to $f - a$ with $a :=\int_{\partial \Omega} f d \sigma_{\partial \Omega}$:
\[
\Var_{\lambda_{\Omega}} f \le \int (f - a)^2 d\lambda_{\Omega} \leq \frac{4}{n^2} \int_{\Omega} \langle x, \nabla f \rangle^2  d \lambda_{\Omega}
+ 2 \Var_{\sigma_{\partial \Omega}} f|_{\partial \Omega}.
\]
This is the desired assertion. 
\end{proof}

\subsection{Optimal Transport to Euclidean Ball}

A remarkable theorem of Y. Brenier \cite{BrenierMap} asserts that between any two absolutely continuous probability measures $\mu,\eta$ on $\Real^n$ (say having second moments), there exists a unique ($\mu$ a.e.) map $T$ which minimizes the transport-cost $\int \abs{T(x) - x}^2 d\mu(x)$, among all maps pushing forward $\mu$ onto $\eta$; moreover, this optimal transport map $T$ is characterized as being the gradient of a convex function $\varphi$. See also \cite{McCannOTOnManifolds} for refinements and extensions. The regularity properties of $T$ have been studied by Caffarelli \cite{CaffarelliHigherHolderRegularity,CaffarelliRegularity,CaffarelliBoundaryRegularity}, who discovered that a necessary condition for $T$ to be smooth is that $\eta$ have convex support; in particular, Caffarelli's results imply that when $\mu = \lambda_{\Omega}$, $\eta = \lambda_{B_2^n}$ and $\partial \Omega$ is smooth, then so is the Brenier map $T_0 := \nabla \varphi_0$ pushing forward $\mu$ onto $\nu$, on the entire closed $\Omega$ (i.e. all the way up to the boundary). By the change-of-variables formula, we obviously have:
\[
\text{Jac} \; T_0 = \text{det} \; dT_0 = \frac{\abs{B_2^n}}{\abs{\Omega}} ~.
\]

\begin{theorem}[Faber--Krahn with Boundary] \label{thm:Hardy2}
 Let $f$ denote a smooth function on $\Omega$. Then:
\begin{equation} \label{eq:Hardy2}
\Var_{\lambda_{\Omega}} f \leq  \frac{4 \abs{\Omega}^{2/n}}{n^2 \abs{B_2^n}^{2/n}} \int_{\Omega} \abs{\nabla f}^2 d \lambda_{\Omega} + \frac{2 \abs{\partial \Omega}}{n \abs{B_2^n}^{1/n} \abs{\Omega}^{(n-1)/n}} \Var_{\lambda_{\partial \Omega}} f|_{\partial \Omega}.
\end{equation}
\end{theorem}
\begin{proof}
Identifying $\Real^n$ with its tangent spaces, we set $\xi = \nabla \varphi_0$ (where $\varphi_0$ was defined above). Note that since $\varphi_0$ is convex, hence $\text{Hess }\varphi_0$ is positive-definite, we may apply the arithmetic-geometric means inequality:
\[
div(\xi) = \Delta \varphi_0 = tr( Hess \varphi_0) \geq n (\text{det} \; Hess \varphi_0)^{1/n} = n (\text{det} \; dT_0)^{1/n} = n \frac{\abs{B_2^n}^{1/n}}{\abs{\Omega}^{1/n}} =: \alpha .
\]
Applying (\ref{eq:BasicHardy}) with $\lambda \equiv \alpha/2$, and using that $\xi = \nabla \varphi_0 \in B_2^n$, we obtain:
\[
\int_{\Omega} f^2 dx \leq \frac{4 \abs{\Omega}^{2/n}}{n^2 \abs{B_2^n}^{2/n}} \int_{\Omega} \abs{\nabla f}^2 dx + \frac{2 \abs{\Omega}^{1/n}}{n \abs{B_2^n}^{1/n}} \int_{\partial \Omega} f^2  d\H^{n-1} .
\]
In particular, when $f$ vanishes on $\partial \Omega$, we deduce (\ref{eq:FK}) with a slightly inferior constant; however, this constant is asymptotically (as $n \rightarrow \infty$) best possible, see Remark \ref{rem:Watson} below. Dividing by $\abs{\Omega}$ and applying the resulting inequality to $f - a$ with $a :=\int_{\partial \Omega} f d\lambda_{\partial \Omega}$, the assertion follows. 
\end{proof}

\begin{remark} \label{rem:Watson}
It is known (e.g. \cite[p. 139]{FollandBook}) that $1 / P^D_{B_2^n}$ is equal to the square of the first positive zero of the Bessel function of order $(n-2)/2$. According to \cite[p. 516]{Watson-BesselBook}, the first zero of the Bessel function of order $\beta$ is $\beta + c_0 \beta^{1/3} + O(1)$, for a constant $c_0 \simeq 1.855$, and so consequently $P^D_{B_2^n} = \frac{4}{n^2} (1 + o(1))$. By homogeneity, it follows that $P^D_{\Omega^*} = 4 \abs{\Omega^*}^{2/n}/(n^2 \abs{B_2^n}^{2/n})(1 + o(1))$, confirming that the constant in Theorem \ref{thm:Hardy2} is asymptotically best possible. 
\end{remark}

\begin{remark}
Note that if we start from (\ref{eq:ByParts}) and avoid employing the Cauchy-Schwartz inequality used to derive (\ref{eq:BasicHardy}), the above proof (using $\xi = \nabla \varphi_0$ and $g \equiv 1$) yields the isoperimetric inequality with sharp constant for smooth bounded domains:
\begin{equation} \label{eq:isop}
\abs{\partial \Omega} \geq n \abs{B_2^n}^{1/n} \abs{\Omega}^{(n-1)/n} .
\end{equation}
This proof was first noted by McCann \cite{McCannConvexityPrincipleForGases}, extending an analogous proof by Knothe and subsequently Gromov of the Brunn-Minkowski inequality \cite{Milman-Schechtman-Book} using the Knothe map \cite{KnotheMap}. See \cite{Figalli-SurveyInMyProceedings} for rigorous extensions of such an approach to non-smooth domains.  
\end{remark}

\subsection{Normal vector field}

In this subsection, we assume in addition that $\Omega$ is strictly convex. We employ the vector field:
\[
\xi(x) = \nu(x / \norm{x})  ~,
\]
the exterior unit normal-field to the convex set $\Omega_x := \norm{x} \Omega$. Note that this field is not well defined (and in particular not continuous) at the origin, so strictly speaking we cannot appeal to (\ref{eq:BasicHardy}). However,  this is not an issue, since $div(\xi)$ is homogeneous of degree $-1$, and so the Jacobian term in polar coordinates $r^{n-1}$ will absorb the blow-up of the divergence near the origin (recall $n \geq 2$). To make this rigorous, we simply repeat the derivation of (\ref{eq:BasicHardy}) by integrating by parts on $\Omega \setminus \eps B_2^n$, and note that we may take the limit as $\eps \rightarrow 0$, since the contribution of the additional boundary $\partial \eps B_2^n$ goes to zero as $\xi$ and $f$ are bounded. 

Now, observe that:
\[
div(\xi)(x) = H_{\partial \Omega_x}(x) = \frac{1}{\norm{x}} H_{\partial \Omega}(x / \norm{x}) ,
\]
where $H_S(y)$ denotes the mean-curvature (trace of the second fundamental form $\text{II}_S$) of a smooth oriented hypersurface $S$ at $x$. Indeed, by definition $\nabla \xi|_{\xi^{\perp}} = \text{II}_{\partial \Omega_x}$, and $2 \nabla_\xi \xi = \nabla \scalar{\xi,\xi} = 0$, and so  $div(\xi) = tr(\nabla \xi) = H_{\partial \Omega_x}$. 

\begin{theorem}[Mean-Curvature Weighted Hardy] \label{thm:Hardy3}
For any strictly convex $\Omega$ and smooth function $f$ defined on it:
\[
\int_{\Omega} \frac{H_{\partial \Omega}(x / \norm{x})}{\norm{x}} f^2(x) dx \leq 4 \int_{\Omega} \frac{\norm{x}}{H_{\partial \Omega}(x / \norm{x})} \scalar{\nabla f(x), \nu(x/\norm{x})}^2 dx + 2 \int_{\partial \Omega} f^2 d\H^{n-1} .
\]
\end{theorem}
\begin{proof}
Immediate after appealing to (\ref{eq:BasicHardy}) with $\lambda(x) = \frac{1}{2} \frac{H_{\partial \Omega}(x / \norm{x})}{\norm{x}}$. 
\end{proof}

\begin{remark} \label{rem:Hardy3}
We note for future reference that by (\ref{eq:ByParts}) with $g \equiv 1$ we have:
\[
\int_{\Omega} \frac{H_{\partial \Omega}(x / \norm{x})}{\norm{x}} d\lambda_{\Omega} = \frac{\abs{\partial \Omega}}{\abs{\Omega}} .
\]
Also, integration in polar coordinates immediately verifies:
\[
\int_{\Omega} \frac{H_{\partial \Omega}(x / \norm{x})}{\norm{x}} d\lambda_{\Omega}(x) = \frac{n}{n-1}  \int_{\partial \Omega} H_{\partial \Omega} d\sigma_{\partial \Omega} ~,~ \int_{\Omega} \frac{\norm{x}}{H_{\partial \Omega}(x / \norm{x})} d\lambda_{\Omega}(x) =\frac{n}{n+1} \int_{\partial \Omega} \frac{d\sigma_{\partial \Omega}}{H_{\partial \Omega}} .
\]
\end{remark}

\subsection{Unconditional Sets}

Finally, we consider one additional vector-field for the Lyapunov method, which is useful when $\Omega$ is the intersection of an unconditional convex set with the first orthant $Q := [0,\infty)^{n}$ under a certain mixed Dirichlet--Neumann boundary condition. Let $int(Q)$ denote the interior of $Q$. 

\begin{theorem}
Let $\Omega \subset Q$ denote a set having smooth boundary, such that
every outer normal $\nu$ to $\partial \Omega \cap int(Q)$ has only non-negative coordinates. Let $f$ denote a smooth function vanishing on $\partial Q$. Then:
\[
\int_{\Omega} \frac{f^2}{\abs{x}^2} dx  \leq \frac{4}{n^2} \int_{\Omega} |\nabla f|^2 dx .
\]
\end{theorem}

\begin{proof}
Consider the vector field:
\[
\xi  = - \brac{\frac{1}{x_1}, \cdots, \frac{1}{x_n}}.
\]
Since $\scalar{\xi, \nu} \leq 0$ in $int(Q) \cap \partial \Omega$ and $f|_{\partial Q \cap \partial \Omega} = 0$, we have:
\begin{align*}
 \int_{\Omega} \sum_{i=1}^{n} \frac{1}{x_i^2}  f^2 d x & = \int_{\Omega} div(\xi) f^2 dx  = -2 \int_{\Omega} f \scalar{\nabla f, \xi} dx + \int_{\partial \Omega} \scalar{\xi , \nu} f^2 d \mathcal{H}^{n-1}
\\&
\le -2 \int_{\Omega} f \scalar{ \nabla f, \xi } dx
\le 2 \sqrt{ \int_{\Omega} \sum_{i=1}^{n} \frac{1}{x_i^2}  f^2 d x } \sqrt{\int_{\Omega} |\nabla f|^2 dx}.
\end{align*}
Finally, by the arithmetic-harmonic means inequality, we obtain:
\[
n^2 \int_{\Omega}  \frac{f^2}{\abs{x}^2} dx  \leq \int_{\Omega} \sum_{i=1}^{n} \frac{1}{x_i^2}  f^2  dx \leq 4 \int_{\Omega} |\nabla f|^2 dx .
\]
\end{proof}

We stress that this result is very similar to the following variant of the Hardy inequality:
$$
\int_{\Omega} \frac{f^2}{|x|^2} dx \le \Bigl(\frac{2}{n-2}\Bigr)^2 \int_{\Omega} |\nabla f|^2 dx,
$$
which holds for any smooth $f$ vanishing on $\partial \Omega$ (see \cite{GM}).

\section{Reduction of KLS conjecture to subclasses of functions}

\subsection{Reduction to the boundary}

Let us now see how the Hardy-type inequalities of the previous section may be used to reduce the KLS conjecture to the behaviour of $1$-Lipschitz functions on the boundary $\partial \Omega$. We remark that we use here the term ``reduction" in a rather loose sense - we obtain a sufficient condition for the KLS conjecture to hold, but we were unable to show that this is also a necessary one. 

Together with (\ref{eq:equiv}), Theorem \ref{thm:Hardy1} immediately yields:
\begin{corollary} \label{cor:Hardy1}
For any smooth convex domain $\Omega$ with barycenter at the origin:
\begin{equation} \label{eq:Cone-Poincare}
 P^N_{\Omega} \leq C P^\infty_{\Omega} \leq C \brac{\frac{4}{n} P_{\Omega}^{Lin} + 2 P^\infty_{\sigma_{\partial \Omega}}} .
\end{equation}
where $C > 0$ is a universal constant. 
\end{corollary}
\begin{proof}
Apply Theorem \ref{thm:Hardy1} to an arbitrary $1$-Lipschitz function $f$, and note that $\int \abs{x}^2 d\lambda_{\Omega} = \sum_{i=1}^n Var_{\lambda_{\Omega}}(x_i) \leq n P_{\Omega}^{Lin}$.
\end{proof}
Consequently, a sufficient criterion for verifying the KLS conjecture is to establish that $P^\infty_{\sigma_{\partial \Omega}} \leq C'$ for any isotropic convex $\Omega$ - a ``weak KLS conjecture for cone measures". This suggests that the most difficult part of the conjecture concerns the behavior of $1$-Lipschitz functions on the boundary.

\medskip

It may be more desirable to work with the Lebesgue measure $\lambda_{\partial \Omega}$ instead of the cone measure $\sigma_{\partial \Omega}$. Since:
\begin{equation} \label{eq:info}
P^{Lin}_{\Omega} \geq P^{Lin}_{\Omega^*} \simeq \abs{\Omega^*}^{2/n} ~,~  |B_2^n|^{1/n} \simeq \frac{1}{\sqrt{n}} ,
\end{equation}
(see e.g. \cite{Milman-Pajor-LK}), Theorem \ref{thm:Hardy2} together with (\ref{eq:equiv}) immediately yields:
\begin{corollary} \label{cor:Hardy2}
For any smooth convex domain $\Omega$ with barycenter at the origin:
\begin{equation} \label{eq:Bdry-Poincare}
 P^N_{\Omega} \leq C_1 P^\infty_{\Omega} \leq C_2 \brac{\frac{4}{n} P_{\Omega}^{Lin} + 2 I P^\infty_{\lambda_{\partial \Omega}}} ~,~ I := \frac{\abs{\partial \Omega}}{n \abs{B_2^n}^{1/n} \abs{\Omega}^{(n-1)/n}} .
\end{equation}
\end{corollary}

Note that for an isotropic convex body, the isoperimetric ratio term $I$  satisfies:
\begin{equation} \label{eq:boundary}
1 \leq I \leq C' \sqrt{n} \abs{\Omega}^{1/n} .
\end{equation}
The left-hand side in fact holds for any arbitrary set $\Omega$ by the sharp isoperimetric inequality (\ref{eq:isop}). The right-hand side follows since when $\Omega$ is convex and isotropic, it is known that $\Omega \supset \frac{1}{C} B_2^n$ (e.g. \cite{Milman-Pajor-LK}). Consequently (see e.g. \cite{Ball-Volume-Ratios}):
\[
\abs{\partial \Omega} = \lim_{\eps \rightarrow 0} \frac{\abs{\Omega + \eps B_2^n} - \abs{\Omega}}{\eps} \leq \lim_{\eps \rightarrow 0} \frac{|\Omega + \eps C \Omega| - \abs{\Omega}}{\eps} = n C  \abs{\Omega} ,
\]
and so (\ref{eq:boundary}) immediately follows. Up to the value of $C'$, the right-hand side is also sharp, as witnessed by the $n$-dimensional cube. Note that by (\ref{eq:info}), $\abs{\Omega}^{1/n} \simeq P^{Lin}_{\Omega^*} \leq P^{Lin}_{\Omega} = 1$ for any isotropic convex body $\Omega$, and so in fact $I \leq C'' \sqrt{n}$. 

\medskip

To avoid the isoperimetric ratio term $I$ which may be too large, we can instead invoke Theorem \ref{thm:Hardy3}:
\begin{corollary} \label{cor:Hardy3}
For any strictly convex smooth domain $\Omega$:
\begin{equation} \label{eq:Bdry-Poincare2}
P^N_{\Omega} \leq C_2 \brac{A^2 + A \frac{\abs{\partial \Omega}}{\abs{\Omega}} P_{\lambda_{\partial \Omega}}^\infty}  ~,~ A := \int_{\partial \Omega} \frac{d\sigma_{\partial K}}{H_{\partial \Omega}} .
\end{equation}
\end{corollary}
Note that by Jensen's inequality and Remark \ref{rem:Hardy3}:
\[
A \frac{\abs{\partial \Omega}}{\abs{\Omega}} \geq \frac{1}{\int_{\partial \Omega} H_{\partial \Omega} d\sigma_{\partial K}} \frac{\abs{\partial \Omega}}{\abs{\Omega}} = \frac{n}{n-1} ,
\]
but perhaps the term $A \frac{\abs{\partial \Omega}}{\abs{\Omega}}$ is nevertheless still more favorable than $I$. 

\medskip
For the proof, we require the following variant of the notion of $P^\infty_{\Omega}$:
\[
P^{1,\infty}_{\Omega} := \sup \set{ \brac{\int \abs{f - med_{\lambda_{\Omega}} f} d\lambda_{\Omega}}^2 \; ; \; \norm{\abs{\nabla f}}_{L^\infty(\lambda_{\Omega})} \leq 1} .
\]
It follows from the results of \cite{EMilman-RoleOfConvexity} that for any convex $\Omega$:
\begin{equation} \label{eq:equiv2}
P^N_{\Omega} \leq C_1 P^{1,\infty}_{\Omega} \leq C_2 P^\infty_{\Omega} \leq C_2 P^N_{\Omega} .
\end{equation}

\begin{proof}[Proof of Corollary \ref{cor:Hardy3}]
By Cauchy-Schwartz:
\[
\brac{\int_{\Omega} \abs{f} d\lambda_{\Omega}}^2 \leq \int_{\Omega} \frac{\norm{x}}{H_{\partial \Omega}(x / \norm{x})} d\lambda_\Omega(x)\int_{\Omega} \frac{H_{\partial \Omega}(x / \norm{x})}{\norm{x}} f^2(x) d\lambda_{\Omega}(x)
\]
Assuming that $f$ is $1$-Lipschitz and invoking Theorem \ref{thm:Hardy3}, it follows that:
\[
\brac{\int_{\Omega} \abs{f} d\lambda_{\Omega}}^2 \leq B \brac{ 4B + 2 \frac{|\partial \Omega|}{|\Omega|} \int_{\partial \Omega} f^2 d\lambda_{\partial \Omega}} ~,~ B:= \int_{\Omega} \frac{\norm{x}}{H_{\partial \Omega}(x / \norm{x})} d\lambda_\Omega(x) .
\]
Applying this to $f - a$ where $a := \int_{\partial \Omega} f d\lambda_{\partial \Omega}$, we obtain:
\[
P^{1,\infty}_{\Omega} \leq B \brac{ 4B + 2 \frac{|\partial \Omega|}{|\Omega|} P^{\infty}_{\lambda_{\partial \Omega}} } .
\]
But $B = \frac{n}{n+1} A$ by Remark \ref{rem:Hardy3}, and so the assertion follows from (\ref{eq:equiv2}). 
\end{proof}

\subsection{A concrete bound}

To control the variance of $1$-Lipschitz functions on the boundary $\partial \Omega$, 
we recall an argument from our previous work \cite{KolesnikovEMilman-Reilly}, where a generalization of the following inequality of A. Colesanti \cite{ColesantiPoincareInequality} was obtained:
\begin{equation} \label{eq:Colesanti}
 \int_{\partial \Omega} H f^2 d\H^{n-1} - \frac{n-1}{n}\frac{\brac{\int _{\partial \Omega} f d\H^{n-1}}^2}{Vol(\Omega)} \leq  \int_{\partial \Omega} \scalar{\II_{\partial \Omega}^{-1} \;\nabla_{\partial \Omega} f,\nabla_{\partial \Omega} f} d\H^{n-1} ~,
\end{equation} 
for any strictly convex $\Omega$ with smooth boundary and smooth function $f$ on $\partial \Omega$.
Applying the Cauchy-Schwartz inequality, we obtain for any $1$-Lipschitz function $f$ with $\int_{\partial \Omega} f d\lambda_{\partial \Omega} = 0$:
\[
\brac{\int_{\partial \Omega} \abs{f - med_{\lambda_{\partial \Omega}} f} d\lambda_{\partial \Omega}}^2 \leq \brac{\int_{\partial \Omega} \abs{f} d\lambda_{\partial \Omega}}^2 \leq \int_{\partial \Omega} \frac{ d\lambda_{\partial \Omega} }{H_{\partial \Omega}}\int_{\partial \Omega} \frac{ d\lambda_{\partial \Omega}}{\kappa_{\partial \Omega}} ,
 \]
 where $\kappa_{\partial \Omega}(x)$ denotes the (positive) minimal principle curvature of $\partial \Omega$ at $x$, so that $\text{II}_{\partial \Omega} \geq \kappa Id$. Consequently, the right-hand-side is an upper bound on $P^{1,\infty}_{\lambda_{\partial \Omega}} $. Using the equivalence (\ref{eq:equiv2}) in a more general Riemannian setting, we were able to deduce in \cite{KolesnikovEMilman-Reilly} that:
\begin{equation} \label{eq:boundary-SG}
(P^{\infty}_{\lambda_{\partial \Omega}} \leq \; ) \; P^N_{\lambda_{\partial \Omega}} \leq C P^{1,\infty}_{\lambda_{\partial \Omega}}  \leq C \int_{\partial \Omega} \frac{ d\lambda_{\partial \Omega} }{H_{\partial \Omega}}\int_{\partial \Omega} \frac{ d\lambda_{\partial \Omega}}{\kappa_{\partial \Omega}} .
 \end{equation}

Plugging this estimate into the estimates of the previous subsection, we obtain:
\begin{theorem} \label{thm:KLS-MainBound}
For $n$ larger than a universal constant and any isotropic strictly convex body $\Omega$ with smooth boundary in $\Real^n$:
\[
P^{N}_{\Omega} \leq C_2 \frac{\abs{\partial \Omega}}{\sqrt{n} \abs{\Omega}^{\frac{n-1}{n}}} \int_{\partial \Omega} \frac{ d\lambda_{\partial \Omega} }{H_{\partial \Omega}}\int_{\partial \Omega} \frac{ d\lambda_{\partial \Omega}}{\kappa_{\partial \Omega}} .
\]
\end{theorem}
\begin{proof}
The easiest option is to invoke Corollary \ref{cor:Hardy2}, but note that Corollary \ref{cor:Hardy1} or \ref{cor:Hardy3} would also work after an appropriate application of Cauchy-Schwartz. Coupled with (\ref{eq:boundary-SG}), it follows that:
\[
P^{N}_{\Omega} \leq C_1 \brac{ \frac{4}{n} P^{Lin}_{\Omega} + \frac{\abs{\partial \Omega}}{\sqrt{n} \abs{\Omega}^{\frac{n-1}{n}}} \int_{\partial \Omega} \frac{ d\lambda_{\partial \Omega} }{H_{\partial \Omega}}\int_{\partial \Omega} \frac{ d\lambda_{\partial \Omega}}{\kappa_{\partial \Omega}} }.
\] 
But since $P^{Lin}_{\Omega} \leq P^N_{\Omega}$, the assertion follows for e.g. $n \geq 8 C_1$. 
\end{proof}

Note that this estimate yields the correct result, up to constants, for the Euclidean ball. 
A concrete class of isotropic convex bodies for which the first term above $\frac{\abs{\partial \Omega}}{\sqrt{n}}$ is upper bounded by a constant, is the class of quadratically uniform convex bodies $\Omega$, since in isotropic position $\Omega \supset c \sqrt{n} B_2^n$ and $\abs{\Omega}^{1/n} \simeq 1$ (see e.g. \cite{KlartagEMilman-2-Convex}). It is not hard to show that when in addition $\Omega \subset C_1\sqrt{n} B_2^n$ - i.e. $\Omega$ is an isotropic quadratically uniform convex body which is isomorphic to a Euclidean ball - then $P^N_{\Omega} \leq C_2$. It would be very interesting to see if the additional assumption $\Omega \subset C_1\sqrt{n} B_2^n$ could be removed by employing the estimate given by Theorem \ref{thm:KLS-MainBound}.

\subsection{Reduction to harmonic functions}

We conclude this section by providing another different reduction of the KLS conjecture: 

\begin{theorem}[Reduction to Harmonic Functions]
There exists a universal constant $C>1$ so that:
\[
P^N_{\Omega} \leq C P^H_{\Omega} ~,~ P^H_{\Omega} := \sup_{h \in H} \frac{\Var_{\lambda_{\Omega}} h}{\int \abs{\nabla h}^2 d\lambda_{\Omega}}  ,
\]
where $H$ denotes the class of harmonic functions $h$ on $\Omega$. In fact, for large enough $n$, one can use $C=2$. 
\end{theorem}
\begin{proof}
Fix an arbitrary smooth function $f$ on $\Omega$, and solve the Poisson equation $\Delta h = 0$, $h|_{\partial \Omega} = f|_{\partial \Omega}$. One has:
\[
\Var_{\lambda_{\Omega}} f \le 2 ( \Var_{\lambda_{\Omega}} (f-h) + \Var_{\lambda_{\Omega}} h).
\]
Since $f-h$ vanishes on $\partial \Omega$, the Faber-Krahn inequalities (\ref{eq:FK}) or (\ref{eq:Hardy2}) imply:
\[
\Var_{\lambda_{\Omega}} (f-h) \leq \frac{4 \abs{\Omega}^{2/n}}{n^2 \abs{B_2^n}^{2/n}}  \int |\nabla f - \nabla h|^2 \ d \lambda_{\Omega} .
\]
It follows that:
\[
\Var_{\lambda_{\Omega}} f  \leq \max \brac{\frac{C_1}{n} \abs{\Omega}^{2/n} , 2 P^H_{\Omega}}  \brac{\int |\nabla f - \nabla h|^2 \ d \lambda_{\Omega} +  \int |\nabla h|^2 \ d\lambda_{\Omega} } .
\]
But since $h$ is harmonic and $(f-h)|_{\partial \Omega}=0$ we have $\int \langle \nabla f - \nabla h, \nabla h \rangle d \lambda_{\Omega} =0 $, and consequently:
\[
\int \bigl( |\nabla f - \nabla h|^2 + |\nabla h|^2 \bigr) d\lambda_{\Omega} = \int |\nabla f|^2 d\lambda_{\Omega} .
\]
It remains to note that since linear functions are harmonic, $P^H_{\Omega} \geq P^{Lin}_{\Omega} \geq P^{Lin}_{\Omega^*} \simeq \abs{\Omega^*}^{2/n}$, concluding the proof.
\end{proof}
 
\begin{remark}
It is not clear to us if it enough to only control the variance of harmonic functions $h$, so that the restriction $h|_{\partial \Omega}$ is $1$-Lipschitz. The reason is that we do not know whether $h$ has bounded Lipschitz constant on the entire $\Omega$, and so we cannot apply (\ref{eq:equiv}). We believe that the latter would be an interesting property of convex domains which is worth investigating. A small observation in this direction is that $\abs{\nabla h}^2$ is subharmonic and hence satisfies the maximum principle, but we do not know how to control the derivative in the normal direction to $\partial \Omega$.
\end{remark}

\section{Transferring Poincar\'e inequalities from $\mu$ to $K_{\mu}$}

Given an absolutely continuous probability measure $\mu$ on $\Real^n$ having upper-semi-continuous density $f$, the following set was considered by K. Ball \cite{Ball-kdim-sections}:
\[
K_{\mu} := \set{ x \in \Real^n \; ; \; \norm{x}_{K_\mu} \leq 1} ~,~ \text{where } ~ \frac{1}{\norm{x}_{K_\mu}} =  \brac{n \int_0^\infty r^{n-1} f(r x) dr}^{1/n} .
\]
Integration in polar coordinates immediately verifies that $\abs{K_\mu} = \norm{\mu} = 1$. 
A remarkable observation of Ball is that when $f$ is log-concave (i.e. $\log f : \Real^n \rightarrow \Real \cup \set{-\infty}$ is concave), then $K_\mu$ is a compact convex set (see \cite{KlartagPerturbationsWithBoundedLK} for the case that $f$ is non-even). If in addition the origin is in the interior of the support of $\mu$, then it will also be in the interior of $K_\mu$ - we will say in that case that $K_\mu$ is a convex body. 

Given a convex body $K$, consider the map $T(x) = \frac{x}{\norm{x}_K}$ where $\norm{\cdot}_K$ denotes the gauge function of $K$ (when $K$ is origin-symmetric, this function defines a norm). It is an elementary exercise to show that $T_{*} \mu = \sigma_{\partial K}$ if and only if $K = c K_\mu$ for some $c>0$ (see \cite[Proposition 3.1]{EMilmanSodinIsoperimetryForULC}). 

\begin{proposition} \label{prop:op}
Let $\mu = f(x) dx$ denote a probability measure with log-concave density on $\Real^n$ ($n \geq 3$) and barycenter at the origin, and set $T(x) = \frac{x}{\norm{x}_{K_\mu}}$.  Assume that $K_\mu \supset R B_2^n$. Then:
\[
\int_{\Real^n} \norm{dT^*(x)}^2_{op} d\mu(x) \leq C \frac{f(0)^{2/n}}{R^2} \int_{K_\mu} \abs{x}^2 dx  ,
\]
where $\norm{\cdot}_{op}$ denotes the operator norm, and $dT^*(x)$ is the dual operator to the differential $dT(x) : T_x \Real^n \rightarrow T_{T(x)} \partial K_{\mu}$. 
\end{proposition}

For the proof, we first require:
\begin{lemma} \label{lem:op}
If $T(x) = \frac{x}{\norm{x}_{K}}$ and $\partial K$ is smooth then:
\[
\norm{dT^*(x)}_{op} = \frac{\abs{x} \abs{\nabla \norm{x}_K} }{\norm{x}^2_K} = \frac{1}{\norm{x}_K \scalar{x/\abs{x} , \nu_{\partial K}(T(x))} } = \frac{\abs{x}}{\norm{x}_K^2 h_K(\nu_{\partial K}(T(x)))} ,
\]
where $h_K(\theta) = \sup\set{ \scalar{x,\theta} ; x \in K}$ denotes the support function of $K$.
\end{lemma}
\begin{proof}
Since $\nabla \norm{x}_K$ is parallel to $\nu = \nu_{\partial K}(T (x))$, taking the partial derivative in the direction of $x$ verifies that:
\[
\nabla \norm{x}_K = \frac{\norm{x}_K}{\scalar{x , \nu}} \nu .
\]
Consequently $dT(x) = \frac{1}{\norm{x}} (Id - \frac{x \otimes \nu}{\scalar{x,\nu}})$. Now observe that:
\[
 \norm{dT(x)^*}_{op}^2 = \sup \set{ \scalar{dT(x) dT(x)^* v , v } \; ; \; v \in T^*_{T(x)} \partial K \; , \;\abs{v} \leq 1}. 
 \]
 But $dT(x) dT(x)^* = \frac{1}{\norm{x}^2}(Id + u \otimes u)$, where $u = \nu - \frac{x}{\scalar{x,\nu}}$. Consequently, its top eigenvalue is:
 \[
 \frac{1}{\norm{x}^2}(1 + |u|^2) =  \frac{1}{\norm{x}^2_K} \frac{\abs{x}^2}{\scalar{x , \nu}^2 } .
 \]
 It remains to note that when $x \in \partial K$ then $\scalar{x,\nu_{\partial K}(x)}$ is precisely the support function of $K$ in the direction of the latter normal. Consequently, $\scalar{x,\nu} =  \norm{x}_K h_K(\nu)$, and the assertion follows. 
\end{proof}

\begin{proof}[Proof of Proposition \ref{prop:op}]
It is easy to see that if the density $f$ of $\mu$ is smooth, then so is $\partial K_\mu$, and so by approximation we may assume that this is indeed the case. 
Consequently, if $K_\mu \supset R B_2^n$, we have by Lemma \ref{lem:op}:
\begin{equation} \label{eq:wasteful}
\int_{\Real^n} \norm{dT^*(x)}^2_{op} d\mu(x) = \int_{\Real^n} \frac{\abs{x}^2}{\norm{x}^4_{K_\mu} h^2_{K_\mu}(\nu_{\partial K_{\mu}}(T(x)))} d\mu(x) \leq \frac{1}{R^2} \int_{\Real^n} \frac{\abs{x}^2}{\norm{x}_{K_\mu}^4} d\mu(x) .  
\end{equation}
Integrating in polar coordinates, we have:
\begin{equation} \label{eq:mid}
\int_{\Real^n} \frac{\abs{x}^2}{\norm{x}_{K_\mu}^4} d\mu(x) = \int_{S^{n-1}} \frac{1}{\norm{\theta}^4_{K_\mu}} \int_0^\infty r^{n-3} f(r \theta) dr d\theta .
\end{equation}
Denoting $k_p(\theta) := (p \int_0^\infty r^{p-1} f(r \theta) dr)^{1/p}$, we use that for any non-negative function $f$ on $[0,\infty)$:
\[
0 < p_1 \leq p_2 \Rightarrow  \frac{k_{p_1}(\theta)}{M_\theta^{1/p_1}} \leq  \frac{k_{p_2}(\theta)}{M_\theta^{1/p_2}} , 
\]
where $M_\theta = \sup_{r \in [0,\infty)} f(r \theta)$. See \cite{BarlowMarshallProschan,Ball-kdim-sections,Milman-Pajor-LK} for case that $f$ is even and \cite[Lemmas 2.5,2.6]{KlartagPerturbationsWithBoundedLK} or \cite[Lemma 3.2 and (3.12)]{PaourisSmallBall} for the general case. Applying this to (\ref{eq:mid}) with $p_1 = n-2$ and $p_2 = n$, denoting $M = \max_{x \in \Real^n} f(x)$, and using polar integration again, it follows that:
\[
\int_{\Real^n} \frac{\abs{x}^2}{\norm{x}_{K_\mu}^4} d\mu(x) \leq M^{2/n} \frac{1}{n-2} \int_{S^{n-1}} \frac{1}{\norm{\theta}_{K_\mu}^{n+2}} d\theta = M^{2/n} \frac{n+2}{n-2} \int_{K_\mu} \abs{x}^2 dx .
\]
It remains to apply a result of M. Fradelizi \cite{FradeliziCentroid} stating that for a log-concave measure $\mu = f(x) dx$ with barycenter at the origin:
\[
M \leq e^n f(0) .
\]
Plugging all of these estimates into (\ref{eq:wasteful}), the assertion is proved. 
\end{proof}

We can now obtain:
\begin{theorem} \label{thm:fvr}
Let $\mu = f(x) dx$ denote a log-concave probability measure on $\Real^n$ having barycenter at the origin. Assume that $K_\mu \supset R B_2^n$. Then for large-enough $n$:
\[
P^N_{K_\mu} \leq C \frac{\int \abs{x}^2 d \lambda_{K_\mu}(x)}{R^2}  f(0)^{2/n} P^N_{\mu}  .
\]
In particular, if $\mu$ satisfies the KLS conjecture then so does $\lambda_{K_\mu}$, as soon as $\frac{\int \abs{x}^2 d \lambda_{K_\mu}}{R^2}$ is bounded above by a constant. 
\end{theorem}

\begin{remark}
This result was already noticed by Bo'az Klartag and the second-named author using a more elaborate computation which was never published. The idea is to control the average Lipschitz constant of the radial map from \cite{EMilmanSodinIsoperimetryForULC} pushing forward $\mu$ onto $\lambda_{K_\mu}$ instead of $\sigma_{\partial K_\mu}$. 
\end{remark}

\begin{proof}
We employ Corollary \ref{cor:Hardy1} and Proposition \ref{prop:op}. When $n$ is large-enough, $C \frac{4}{n} P^{Lin}_{K_\mu} \leq \frac{1}{2} P^{Lin}_{K_\mu} \leq \frac{1}{2} P^N_{K_\mu}$, and hence by Corollary \ref{cor:Hardy1}:
\[
P^N_{K_\mu} \leq C' P^\infty_{\sigma_{\partial K_\mu}} .
\]
Denoting $T(x) = \frac{x}{\norm{x}_{K_\mu}}$, we see by Proposition \ref{prop:op} that for any $1$-Lipschitz function $f$ on $\partial K_\mu$:
\begin{eqnarray*}
& & \Var_{\sigma_{\partial K_\mu}}(f) = \Var_{\mu}(f \circ T) \leq P^N_\mu \int \abs{\nabla (f \circ T)}^2 d\mu \\
&\leq &  P^N_\mu \int \abs{\nabla_{\partial K_\mu} f}^2  (T(x)) \norm{dT^*(x)}^2_{op} d\mu(x) \\
&\leq &  P^N_\mu \int \norm{dT^*(x)}^2_{op} d\mu(x) \leq C \frac{f(0)^{2/n}}{R^2} \int_{K_\mu} \abs{x}^2 dx \; P^N_\mu .
\end{eqnarray*}
This implies the first part of the assertion. 

The second part follows since, as shown by Ball \cite{Ball-kdim-sections} (see \cite{KlartagPerturbationsWithBoundedLK} for the non-even case):
\begin{equation} \label{eq:Lin-equiv}
P^{Lin}_{K_\mu} \simeq f(0)^{2/n} P^{Lin}_\mu .
\end{equation}
 Consequently:
\[
P^N_\mu \leq A P^{Lin}_\mu \;\;\; \Rightarrow \;\;\; P^N_{K_\mu} \leq C \frac{\int \abs{x}^2 d \lambda_{K_\mu}}{R^2} \; A \; P^{Lin}_{K_\mu} .
\]
\end{proof}

We thus obtain a simple recipe for obtaining good spectral-gap estimates on certain convex bodies $K$ having in-radius $R$ so that $\int \abs{x}^2 d \lambda_{K}(x) / R^2$ is bounded above by a constant: if we can find a log-concave measure $\mu$ having good spectral-gap so that $K_\mu = K$, Theorem \ref{thm:fvr} will imply that $K$ also has good spectral-gap. 

\begin{remark}
An inspection of the proofs of Proposition \ref{prop:op} and Theorem \ref{thm:fvr} shows that we may replaces $\frac{1}{R^2}$ in all of the occurrences above, with the more refined expression $\int_{\partial K_{\mu}} \frac{d\sigma_{\partial K_\mu}}{h^2_{K_\mu}(\nu_{\partial K_\mu})}$. However, we do not know how to effectively control the latter quantity. 
\end{remark}

\subsection{An example: unit-balls of $\ell_p^n$, $p \in [1,2]$}

We illustrate this for unit-balls $B_p^n$ of $\ell_p^n$, $p \in [1,2]$. It was first shown by S. Sodin \cite{SodinLpIsoperimetry} that these convex bodies satisfy the KLS conjecture. An alternative derivation was obtain in \cite{EMilman-RoleOfConvexity} by using the weaker $P^\infty$ parameter and the equivalence (\ref{eq:equiv}). Both approaches relied on the Schechtman--Zinn concentration estimates for these bodies \cite{SchechtmanZinn2}. 

Using Theorem \ref{thm:fvr}, we avoid passing through the Schechtman--Zinn concentration results. Indeed, let $\mu_p$ denote the one-dimensional probability measure $\frac{1}{2 \Gamma(1/p + 1)} \exp(-\abs{t}^p) dt$. The $n$-fold product measure $\mu_p^n :=\mu_p^{\otimes n}$ has density $f_p^n(x)$ where:
\[
f_p^n(x) =\frac{1}{2^n \Gamma(1/p + 1)^n} \exp(- \sum_{i=1}^n \abs{x_i}^p) .
\]
By the tensorization property of the Poincar\'e inequality \cite{Ledoux-Book}, $P^N_{\mu_p^n} = P^N_{\mu_p}$, and since any one-dimensional log-concave measure satisfies the KLS conjecture, then so does any log-concave product measure. 
Now, since all level sets of $f_p^n$ are homothetic copies of $B_p^n$, it is immediate to see that $K_{\mu_p^n}$ must be (the necessarily volume one) homothetic copy $\tilde{B}_p^n$ of $B_p^n$. In the range $p \in [1,2]$, it is known (e.g. \cite{Milman-Schechtman-Book}) and easy to check that $\tilde{B}_p^n$ are finite volume-ratio bodies, meaning that $\tilde{B}_p^n \supset c \sqrt{n} B_2^n$. On the other hand, by (\ref{eq:Lin-equiv}): 
\[
\int \abs{x}^2 \lambda_{\tilde{B}_p^n} \simeq f_p^n(0)^{2/n} \int \abs{x}^2 d\mu_p^n(x) = \frac{1}{2^2 \Gamma(1/p + 1)^2} n \int_{-\infty}^\infty |t|^2 d\mu_p(t) \leq C n,
\]
uniformly in $p \in [1,2]$. Consequently, Theorem \ref{thm:fvr} implies that $\tilde{B}_p^n$ (and hence $B_p^n$) satisfy the KLS conjecture, uniformly in $n$ and $p \in [1,2]$. Similar versions may easily be obtained for convex functions more general than $|t|^p$ ; we leave this to the interested reader. 

\subsection{Another example: unit-balls of $\ell_p^n$, $p \in (2,\infty)$}

To conclude, we use the unit-balls of $\ell_p^n$ for $p \in (2,\infty)$ to further illustrate the advantage and disadvantage of the method we propose in this section. Note that $\tilde{B}_p^n$ are not finite volume-ratio bodies when $p \in (2,\infty]$, and so Theorem \ref{thm:fvr} does not directly apply. However, by inspecting its proof and avoiding using the wasteful bound (\ref{eq:wasteful}), we can still deduce the KLS conjecture for these bodies when $p$ is bounded away from $\infty$. It was first shown by R.~Latala and J.O.~Wojtaszczyk \cite{LatalaJacobInfConvolution} that in the entire range $p \in [2,\infty]$, there exists a globally Lipschitz map pushing forward $\mu_p^n$ onto $\lambda_{\tilde{B}_p^n}$, different from the radial map we have considered in this section. It is interesting to note that the radial-map is nevertheless Lipschitz on-average, at least when $p < \infty$. 

Indeed, by inspecting the proof of Theorem \ref{thm:fvr} and employing Lemma \ref{lem:op}, we see that we just need to control: 
\[
\int_{\Real^n} \norm{dT^*(x)}^2_{op} d\mu_p^n(x) = \int_{\Real^n} \frac{\abs{x}^2 |\nabla \norm{x}_{\tilde{B}_p^n}|^2 }{\norm{x}^4_{\tilde{B}_p^n}} d\mu_p^n(x) = c_{p,n}^2 \int_{\Real^n} \frac{\abs{x}^2 |\nabla \norm{x}_{p}|^2 }{\norm{x}^4_{p}} d\mu_p^n(x) ,
\]
where $\tilde{B_p^n} = c_{p,n} B_p^n$.  It is well-known and easy to calculate that $c_{p,n} \simeq n^{1/p}$. 
Using that $\abs{x}^2 \leq n^{1-2/p} \norm{x}_{p}^2$ (since $p \geq 2$), that:
\[
 |\nabla \norm{x}_p|^2 = \frac{\sum_{i=1}^n \abs{x_i}^{2p-2}}{\norm{x}_p^{2p-2}} ,
 \]
 and the invariance under permutation of coordinates, 
 we conclude that:
\[
\int_{\Real^n} \norm{dT^*(x)}^2_{op} d\mu_p^n(x) \simeq n^{2} \int \frac{\abs{x_1}^{2p-2}}{\norm{x}_{p}^{2p}} d\mu_p^n(x) .
\]
Integrating by parts, we have:
\begin{eqnarray*}
& & \int_{\Real^n}\frac{\exp(-\norm{x}_p^p)}{\norm{x}_p^{2p}} \abs{x_1}^{2p-2} dx = \int_{\Real^n} \int_{\norm{x}_p^p}^\infty \exp(-t) \brac{\frac{1}{t^2} + \frac{2}{t^3}} dt \abs{x_1}^{2p-2} dx \\
&= &  \int_0^\infty \exp(-t) \brac{\frac{1}{t^2} + \frac{2}{t^3}} \int_{t^{1/p} B_p^n} \abs{x_1}^{2p-2} dx \; dt =
 \int_0^\infty \exp(-t) \brac{\frac{1}{t^2} + \frac{2}{t^3}} t^{\frac{n+2p-2}{p}} dt \int_{B_p^n} \abs{x_1}^{2p-2} dx ,
\end{eqnarray*}
and so by a similar computation we conclude:
\[
\int \frac{\abs{x_1}^{2p-2}}{\norm{x}_{p}^{2p}} d\mu_p^n(x)  = A B ~,~ A := \frac{\int_0^\infty \exp(-t) \brac{\frac{1}{t^2} + \frac{2}{t^3}} t^{\frac{n+2p-2}{p}} dt}{\int_0^\infty \exp(-t) t^{\frac{n+2p-2}{p}} dt} ~,~ B :=\int \abs{x_1}^{2p-2} d\mu_p^n(x) .
\]
Now:
\[
B = \int_{-\infty}^\infty \abs{t}^{2p-2} d\mu_p(t) = \frac{\Gamma(-1/p)}{p \Gamma(1+1/p)} \leq \frac{C_1}{p} ,
\]
uniformly in $p \in [2,\infty]$, whereas it is elementary to verify that in that range:
\[
A \leq C_2 \min\brac{\frac{p^2}{n^2}, \frac{p}{n}} .
\]
Putting everything together, we see that:
\begin{equation} \label{eq:truth}
\int_{\Real^n} \norm{dT^*(x)}^2_{op} d\mu_p^n(x) \leq C \min(p , n) .
\end{equation}
Consequently, the same argument as in the previous subsection shows that $\tilde{B}_p^n$ verify the KLS conjecture uniformly in $n$, as long as $p$ is bounded above. 

\medskip

It is natural to wonder whether the only inequality we have used to derive the above estimate, namely $\abs{x}^2 \leq n^{1-2/p} \norm{x}_{p}^2$, was perhaps too crude. However, this is not the case, and unfortunately it is the method of working with the map $T(x) = x / \norm{x}_{K_\mu}$ which is too crude. Indeed, when $p=\infty$, so that $\mu_\infty^n$ is the uniform measure on $[-1,1]^n$ and $K = K_{\mu_\infty^n} = [-1/2,1/2]^n$, we see by Lemma \ref{lem:op} that:
\[
\norm{dT^*(x)}_{op} = \frac{\abs{x} \abs{\nabla \norm{x}_K} }{\norm{x}^2_K} = \frac{\abs{x}}{4 \norm{x}^2_\infty} ,
\]
and consequently:
\[
\int_{\Real^n} \norm{dT^*(x)}^2_{op} d\mu_p^n(x)  \simeq n ,
\]
confirming that our estimate (\ref{eq:truth}) is tight. This example suggests that perhaps it is better to work with the radial map from \cite{EMilmanSodinIsoperimetryForULC} pushing forward $\mu$ onto $\lambda_{K_\mu}$ instead of our map $T$ which pushes $\mu$ onto $\sigma_{\partial K_\mu}$.

\def\cprime{$'$} \def\textasciitilde{$\sim$}


\begin{thebibliography}{10}

\bibitem{Ball-kdim-sections}
K.~Ball.
\newblock Logarithmically concave functions and sections of convex sets in
  $\mathbb{R}^n$.
\newblock {\em Studia Math.}, 88(1):69--84, 1988.

\bibitem{Ball-Volume-Ratios}
K.~Ball.
\newblock Volume ratios and a reverse isoperimetric inequality.
\newblock {\em J. London Mathematical Society}, 44(2):351--359, 1991.

\bibitem{BarlowMarshallProschan}
R.~E. Barlow, A.~W. Marshall, and F.~Proschan.
\newblock Properties of probability distributions with monotone hazard rate.
\newblock {\em Ann. Math. Statist.}, 34:375--389, 1963.

\bibitem{Benguria-Lambda1Survey}
R.~D. Benguria.
\newblock Isoperimetric inequalities for eigenvalues of the {L}aplacian.
\newblock In {\em Entropy and the quantum {II}}, volume 552 of {\em Contemp.
  Math.}, pages 21--60. Amer. Math. Soc., Providence, RI, 2011.

\bibitem{BrenierMap}
Y.~Brenier.
\newblock Polar factorization and monotone rearrangement of vector-valued
  functions.
\newblock {\em Comm. Pure Appl. Math.}, 44(4):375--417, 1991.

\bibitem{CaffarelliHigherHolderRegularity}
L.~A. Caffarelli.
\newblock Interior {$W^{2,p}$} estimates for solutions of the
  {M}onge-{A}mp\`ere equation.
\newblock {\em Ann. of Math. (2)}, 131(1):135--150, 1990.

\bibitem{CaffarelliBoundaryRegularity}
L.~A. Caffarelli.
\newblock Boundary regularity of maps with convex potentials.
\newblock {\em Comm. Pure Appl. Math.}, 45(9):1141--1151, 1992.

\bibitem{CaffarelliRegularity}
L.~A. Caffarelli.
\newblock The regularity of mappings with a convex potential.
\newblock {\em J. Amer. Math. Soc.}, 5(1):99--104, 1992.

\bibitem{CattiauxGuillin-LyapunovSurvey}
P.~Cattiaux and A.~Guillin.
\newblock Functional inequalities via Lyapunov conditions.
\newblock In {\em Proceedings of the Summer School on Optimal Transport
  (Grenoble 2009)}.
\newblock arXiv:1001.1822.

\bibitem{ColesantiPoincareInequality}
A.~Colesanti.
\newblock From the {B}runn-{M}inkowski inequality to a class of
  {P}oincar\'e-type inequalities.
\newblock {\em Commun. Contemp. Math.}, 10(5):765--772, 2008.

\bibitem{Eldan-StochasticLocalization}
R.~Eldan.
\newblock Thin shell implies spectral gap up to polylog via a stochastic
  localization scheme.
\newblock {\em Geom. Funct. Anal.}, 23(2):532--569, 2013.

\bibitem{Figalli-SurveyInMyProceedings}
A.~Figalli.
\newblock Quantitative isoperimetric inequalities, with applications to the
  stability of liquid drops and crystals.
\newblock In {\em Concentration, functional inequalities and isoperimetry},
  volume 545 of {\em Contemp. Math.}, pages 77--87. Amer. Math. Soc.,
  Providence, RI, 2011.

\bibitem{FollandBook}
G.~B. Folland.
\newblock {\em Introduction to partial differential equations}.
\newblock Princeton University Press, Princeton, NJ, second edition, 1995.

\bibitem{FradeliziCentroid}
M.~Fradelizi.
\newblock Sections of convex bodies through their centroid.
\newblock {\em Arch. Math. (Basel)}, 69(6):515--522, 1997.

\bibitem{GM}
N.~Ghoussoub, A.~Moradifam, 
\newblock  Functional inequalities: new perspectives and new applications,
\newblock vol. 187 of {\em Math. Surv. Monographs}, Amer. Math. Soc., 2013.

\bibitem{GuedonEMilmanInterpolating}
O.~Gu{\'{e}}don and E.~Milman.
\newblock Interpolating thin-shell and sharp large-deviation estimates for
  isotropic log-concave measures.
\newblock {\em Geom. Func. Anal.}, 21(5):1043--1068, 2011.

\bibitem{KLS}
R.~Kannan, L.~Lov{\'a}sz, and M.~Simonovits.
\newblock Isoperimetric problems for convex bodies and a localization lemma.
\newblock {\em Discrete Comput. Geom.}, 13(3-4):541--559, 1995.

\bibitem{KlartagPerturbationsWithBoundedLK}
B.~Klartag.
\newblock On convex perturbations with a bounded isotropic constant.
\newblock {\em Geom. and Funct. Anal.}, 16(6):1274--1290, 2006.

\bibitem{KlartagEMilman-2-Convex}
B.~Klartag and E.~Milman.
\newblock On volume distribution in $2$-convex bodies.
\newblock {\em Israel J. Math.}, 164:221--249, 2008.

\bibitem{KnotheMap}
H.~Knothe.
\newblock Contributions to the theory of convex bodies.
\newblock {\em Michigan Math. J.}, 4:39--52, 1957.

\bibitem{KolesnikovEMilman-Reilly}
A.~V. Kolesnikov and E.~Milman.
\newblock {P}oincar\'e and {B}runn-{M}inkowski inequalities on weighted
  manifolds with boundary.
\newblock submitted, arxiv.org/abs/1310.2526, 2014.

\bibitem{LatalaJacobInfConvolution}
R.~Lata{\l}a and J.~O. Wojtaszczyk.
\newblock On the infimum convolution inequality.
\newblock {\em Studia Math.}, 189(2):147--187, 2008.

\bibitem{Ledoux-Book}
M.~Ledoux.
\newblock {\em The concentration of measure phenomenon}, volume~89 of {\em
  Mathematical Surveys and Monographs}.
\newblock American Mathematical Society, Providence, RI, 2001.

\bibitem{McCannConvexityPrincipleForGases}
R.~J. McCann.
\newblock A convexity principle for interacting gases.
\newblock {\em Adv. Math.}, 128(1):153--179, 1997.

\bibitem{McCannOTOnManifolds}
R.~J. McCann.
\newblock Polar factorization of maps on {R}iemannian manifolds.
\newblock {\em Geom. Funct. Anal.}, 11(3):589--608, 2001.

\bibitem{EMilman-RoleOfConvexity}
E.~Milman.
\newblock On the role of convexity in isoperimetry, spectral-gap and
  concentration.
\newblock {\em Invent. Math.}, 177(1):1--43, 2009.

\bibitem{EMilmanSodinIsoperimetryForULC}
E.~Milman and S.~Sodin.
\newblock An isoperimetric inequality for uniformly log-concave measures and
  uniformly convex bodies.
\newblock {\em J. Funct. Anal.}, 254(5):1235--1268, 2008.

\bibitem{Milman-Pajor-LK}
V.~D. Milman and A.~Pajor.
\newblock Isotropic position and interia ellipsoids and zonoids of the unit
  ball of a normed $n$-dimensional space.
\newblock In {\em Geometric Aspects of Functional Analysis}, volume 1376 of
  {\em Lecture Notes in Mathematics}, pages 64--104. Springer-Verlag,
  1987-1988.

\bibitem{Milman-Schechtman-Book}
V.~D. Milman and G.~Schechtman.
\newblock {\em Asymptotic theory of finite-dimensional normed spaces}, volume
  1200 of {\em Lecture Notes in Mathematics}.
\newblock Springer-Verlag, Berlin, 1986.
\newblock With an appendix by M. Gromov.

\bibitem{PaourisSmallBall}
G.~Paouris.
\newblock Small ball probability estimates for log-concave measures.
\newblock {\em Trans. Amer. Math. Soc.}, 364(1):287--308, 2012.

\bibitem{SchechtmanZinn2}
G.~Schechtman and J.~Zinn.
\newblock Concentration on the {$l\sp n\sb p$} ball.
\newblock In {\em Geometric aspects of functional analysis}, volume 1745 of
  {\em Lecture Notes in Math.}, pages 245--256. Springer, Berlin, 2000.

\bibitem{SodinLpIsoperimetry}
S.~Sodin.
\newblock An isoperimetric inequality on the {$\ell_p$} balls.
\newblock {\em Ann. Inst. H. Poincar\'e Probab. Statist.}, 44(2):362--373,
  2008.

\bibitem{Watson-BesselBook}
G.~N. Watson.
\newblock {\em A treatise on the theory of {B}essel functions}.
\newblock Cambridge Mathematical Library. Cambridge University Press,
  Cambridge, 1995.
\newblock Reprint of the second (1944) edition.

\end{thebibliography}
\end{document}